\documentclass[11pt,reqno]{amsart} 

\usepackage{graphicx}

\usepackage[top=1in,bottom=1in,left=1in,right=1in]{geometry}

\usepackage{xcolor}

\usepackage{enumitem}

\usepackage[labelfont=bf,font=md]{caption}

\usepackage[font={sf,small},labelfont=md]{subfig}

\usepackage[noadjust,nosort]{cite}

\usepackage{amsmath, amsthm, amssymb}
\usepackage{enumerate}
\usepackage{tikz}
\usepackage{mathtools}
\usepackage{array}

\usepackage{polynom}

\usepackage{multicol}

\usepackage{dsfont}
	\newcommand{\one}{\mathds{1}}

\usepackage{framed}

\usepackage{hyperref}
	\hypersetup{colorlinks=true,linkcolor=blue,citecolor=blue,urlcolor=black}
	
\allowdisplaybreaks

\numberwithin{equation}{section}


\newcommand{\eq}[1]{\begin{align*} #1 \end{align*}}
\newcommand{\eeq}[1]{\begin{align} \begin{split} #1 \end{split} \end{align}}

\newcommand{\E}{\mathbb{E}}
\renewcommand{\P}{\mathbb{P}}
\newcommand{\R}{\mathbb{R}}
\newcommand{\Z}{\mathbb{Z}}

\renewcommand{\b}{\mathcal{B}}
\renewcommand{\SS}{\mathcal{S}}
\newcommand{\f}{\mathcal{F}}
\renewcommand{\l}{\mathcal{L}}

\newtheorem{thm}{Theorem}[section]

\newtheorem{cor}[thm]{Corollary}
\newtheorem{lemma}[thm]{Lemma}

\theoremstyle{definition}
\newtheorem{defn}[thm]{Definition}

\newtheorem{remark}[thm]{Remark}

\def\eps{\varepsilon}

\newcommand{\vc}[1]{{\boldsymbol #1}}

\newcommand{\wt}[1]{\widetilde{#1}}

\DeclareMathOperator{\Var}{Var}

\DeclareMathOperator*{\esssup}{ess\,sup}
\DeclareMathOperator*{\essinf}{ess\,inf}

\newcommand{\givenp}[3][]{#1( #2 \: #1| \: #3 #1)} 

\DeclareMathOperator{\e}{e}

\DeclareMathOperator{\tv}{TV}

\newcommand{\dd}{\mathrm{d}}

\DeclarePairedDelimiter\ceil{\lceil}{\rceil}
\DeclarePairedDelimiter\floor{\lfloor}{\rfloor}

\DeclareMathOperator*{\argmin}{arg\,min} 

\renewcommand{\thefootnote}{\fnsymbol{footnote}}

\title{Fluctuation lower bounds in planar random growth models}



\author{Erik Bates}
\thanks{Erik Bates's research was partially supported by NSF grant DGE-114747} 
\address{\newline Department of Mathematics \newline University of California, Berkeley \newline 1067 Evans Hall \newline Berkeley, CA 94720-3840 \newline \textup{\tt ewbates@berkeley.edu}}

\author{Sourav Chatterjee}
\thanks{Sourav Chatterjee's research was partially supported by NSF grant DMS-1608249}
\address{\newline Department of Statistics \newline Stanford University\newline Sequoia Hall, 390 Jane Stanford Way \newline Stanford, CA 94305-4020\newline \textup{\tt souravc@stanford.edu}}

\begin{document}
\bibliographystyle{acm}

\renewcommand{\thefootnote}{\arabic{footnote}} \setcounter{footnote}{0}

\begin{abstract}
We prove $\sqrt{\log n}$ lower bounds on the order of growth fluctuations in three planar growth models (first-passage percolation, last-passage percolation, and directed polymers) under no assumptions on the distribution of vertex or edge weights other than the minimum conditions required for avoiding pathologies. 
Such bounds were previously known only for certain restrictive classes of distributions.
In addition, the first-passage shape fluctuation exponent is shown to be at least $1/8$, extending previous results to more general distributions. 
\end{abstract}

\maketitle

\section{Introduction}
Even after years of study on random growth models, such as first- and last-passage percolation and directed polymers, much remains mysterious or out of reach technically.
For instance, beyond the fundamental shape theorems guaranteeing linear growth rates for the passage times/free energy, there are sublinear fluctuations whose asymptotics are not established.
Even in the planar setting, for which the conjectural picture is clear, general tools are far from making it rigorous.
This is in stark contrast with integrable models, for which fluctuation exponents are only a fraction of what has been proved.
In this paper we consider three widely studied random growth models: first-passage percolation (FPP), last-passage percolation (LPP), and directed polymers in random environment.
While the models differ in how growth is measured, they each possess a law of large numbers that says the rate of growth is asymptotically linear.
More mysterious, however, are the sublinear fluctuations.
In their two-dimensional versions, these models are believed to belong to the Kardar--Parisi--Zhang universality class \cite{corwin12}, and in particular that growth fluctuations are of order $n^{1/3}$.
Except in exceptional cases of LPP and directed polymers having exact solvability properties, rigorous results are far from this goal, or in some cases non-existent.

The goal of this article is two-fold.
First, we describe a general strategy for proving lower bounds on the order of fluctuations for a sequence of random variables (defined precisely in Definition \ref{fluctuations_def}). The approach is an adaptation of techniques developed recently by the second author in \cite{chatterjee19II}. It is general in that it can be used in a wide variety of problems consisting of i.i.d.~random variables, where no assumptions are made on the common distribution of these variables.
Second, we apply the method to study fluctuations in the growth of planar FPP, LPP, and directed polymers.
In all three cases, we are able to prove a lower bound of order $\sqrt{\log n}$ fluctuations.
In addition, for FPP we extend the shape fluctuation lower bound of $n^{1/8-\delta}$ to almost all distributions for which it should be true.
Although still far from $n^{1/3}$, which by all accounts is the correct order (e.g.~see \cite{sosoe18} and references therein), our results require almost no assumptions on the underlying weight distribution.

The paper is structured as follows.
The general method mentioned above for establishing fluctuation lower bounds is outlined in Section \ref{fluctuations_sec}, and some necessary lemmas are proved.
The random growth models under consideration are introduced in Section \ref{results}, where the main results are also stated.
Finally, Section \ref{proofs} sees the method put into action to prove these results.

\section{General method for lower bounds on fluctuations} \label{fluctuations_sec}

\subsection{Definitions}

Let us begin by precisely stating what is meant by a lower bound on fluctuations.

\begin{defn} \label{fluctuations_def}
Let $(X_n)_{n\geq1}$ be a sequence of random variables, and let $(\delta_n)_{n\geq1}$ be a sequence of positive real numbers.
We will say that $X_n$ has fluctuations of order at least $\delta_n$ if there are positive constants $c_1$ and $c_2$ such that for all large $n$, and for all $-\infty < a \leq b < \infty$ with $b-a \leq c_1\delta_n$, one has $\P(a\leq X_n\leq b) \leq 1 - c_2$.
\end{defn}

In other words, fluctuations are of order at least $\delta_n$ if no sequence of intervals $I_n$ of length $o(\delta_n)$ satisfies $\P(X_n\in I_n) \to 1$.
Note that if fluctuations are at least of order $\delta_n$, then so is $\sqrt{\Var(X_n)}$.
The converse, however, is not true in general, necessitating alternative approaches even when a lower bound on variance is known.
On the other hand, if a variance lower bound is accompanied by an upper bound of the same order, then fluctuations must be of that order.
One can see this from a second moment argument, for instance using the Paley--Zygmund inequality.
In the absence of matching variance bounds, one must work with Definition \ref{fluctuations_def} directly.
For this reason, the following simple lemma is useful.

\begin{lemma}[{\cite[Lemma 1.2]{chatterjee19II}}] \label{fluctuation_lemma}
Let $X$ and $Y$ be random variables defined on the same probability space.
For any $-\infty < a \leq b < \infty$,
\eq{
\P(a \leq X \leq b) \leq \frac{1}{2}\Big(1 + \P(|X-Y| \leq b-a) + d_{\tv}(\l_X,\l_{Y})\Big),
} 
where $\l_X$ and $\l_Y$ denote the laws of $X$ and $Y$, respectively.
\end{lemma}

Here $d_{\tv}(\nu_1,\nu_2)$ is the total variation distance between probability measures $\nu_1,\nu_2$ on the same measurable space $(\Omega,\f)$, defined as
\eq{
d_{\tv}(\nu_1,\nu_2) \coloneqq \sup_{A \in \f} |\nu_1(A) - \nu_2(A)|.
}
It can be related to Hellinger affinity between $\mu$ and $\wt\mu$,
\eeq{ \label{hellinger_def}
\rho(\nu_1,\nu_2) \coloneqq \int_\Omega \sqrt{fg}\ \dd\nu_0,
}
where $\nu_0$ is any probability measure on $(\Omega,\f)$ with respect to which both $\nu_1$ and $\nu_2$ are absolutely continuous, and $f$ and $g$ are their respective densities.
Since
\eq{
d_{\tv}(\nu_1,\nu_2) = \frac{1}{2}\int_\Omega |f-g|\ \dd\nu_0,
}
the following upper bound follows from the Cauchy--Schwarz inequality:
\eeq{ \label{tv_hellinger}
d_{\tv}(\nu_1,\nu_2) \leq \sqrt{1 - \rho(\nu_1,\nu_2)^2}.
}

\subsection{The general method}
To produce a lower bound on the order of fluctuations using Lemma \ref{fluctuation_lemma}, the basic idea is to introduce a coupling $(X,Y)$ such that $|X-Y|$ is large with substantial probability while $d_{\tv}(\l_X,\l_Y)$ is small.
A general approach formalizing this idea was initiated in \cite{chatterjee19II}, in which the couplings are obtained from multiplicative perturbations inspired by the Mermin--Wagner theorem of statistical mechanics~\cite{mermin-wagner66}.
Such couplings only work, however, for a certain class of random variables, namely those with \eeq{ \label{previous_assumption}
&\text{density proportional to $\e^{-V}$, where $V \in C^\infty(\R)$, such that} \\
&\text{$V$ and its derivatives of all orders have at most polynomial growth, and} \\
&\text{$\e^V$ grows faster than any polynomial.}
}
We now propose a different type of coupling that allows for the approach of \cite{chatterjee19II} to be extended to any distribution.
Although the couplings we will use to prove the main theorems of this paper are more specific, we present here the most general setup in hopes that the method might be useful in other settings.

Consider a real-valued random variable $X$ defined on some probability space $(\Omega,\f,\P)$.
Let $\l_X$ denote the law of $X$.
Suppose $X'$ is another random variable defined on the same probability space, such that
$\l_{X'}$ is absolutely continuous with respect to $\l_X$ and has bounded density.
Given $\eps \in (0,1)$, let $Y$ be a Bernoulli($\eps$) random variable independent of $X$ and $X'$.
Finally, set
\eeq{ \label{new_X}
\wt X = \begin{cases}
X' &\text{if }Y=1, \\
X &\text{if }Y=0.
\end{cases}
}

\begin{lemma} \label{hellinger_lemma}
The Hellinger affinity between $\l_X$ and $\l_{\wt X}$ satisfies the lower bound
\eq{
\rho(\l_X,\l_{\wt X}) \geq 1 - C\eps^2,
}
where $C$ is a constant depending only on $\l_X$ and $\l_{X'}$. 
\end{lemma}

\begin{proof}
Let us denote the density of $\l_{X'}$ with respect to $\l_{X}$ by $f(t)$, which we assume to be bounded;
say $f(t) \leq M$.
It is easy to see that $\eps f(t) + 1-\eps$ is the density of $\l_{\wt X}$ with respect to $\l_{X}$, and so
\eq{
\rho(\l_X,\l_{\wt X}) = \int_\R \sqrt{\eps f(t) + 1-\eps}\ \l_X(\dd t).
}
For $\eps < 1/M$, we can write the Taylor expansion
\eq{
\sqrt{1-\eps[1-f(t)]} = 1 - \frac{\eps}{2}[1-f(t)] - \frac{\eps^2}{8}[1-f(t)]^2 + \eps^3r(t),
}
where $r(t)$ is bounded.
In fact, the entire right-hand side above is bounded, and so there is no problem in writing
\eq{
\rho(\l_X,\l_{\wt X}) &= \int_\R\Big(1 - \frac{\eps}{2}[1-f(t)] - \frac{\eps^2}{8}[1-f(t)]^2 + \eps^3r(t)\Big)\ \l_X(\dd t).
}
Using the fact that $\int_\R f(t)\, \l_X(\dd t) = 1$, we find
\eq{
\rho(\l_X,\l_{\wt X}) &= 1 - \frac{\eps^2}{8}\int_\R[1-f(t)]^2\ \l_X(\dd t) + O(\eps^3) \geq 1 - C\eps^2,
}
where $C$ depends only on $\l_X$ and $\l_{X'}$.
Replacing $C$ by $\max(C,M^2)$ allows the statement to also hold trivially for $\eps \geq 1/M$.
\end{proof}

When the same type of coupling is applied to several i.i.d.~variables, we get the following bound which can be used in Lemma \ref{fluctuation_lemma}.

\begin{lemma} \label{tv_lemma}
Let $X_1,\dots,X_n$ be i.i.d.~random variables with law $\l_X$, and $X_1',\dots,X_n'$ be $i.i.d.$~random variables with law $\l_{X'}$. 
Assume $\l_{X'}$ is absolutely continuous with respect to $\l_X$ with bounded density.
For each $i = 1,\dots,n$, let $Y_i$ be a $\mathrm{Bernoulli}(\eps_i)$ random variable independent of everything else, and define $\wt X_i$ as in \eqref{new_X} with $\eps = \eps_i$.
Then
\eq{
d_{\tv}(\l_{(X_1,\dots,X_n)},\l_{(\wt X_1,\dots,\wt X_n)}) \leq C\bigg(\sum_{i=1}^n \eps_i^2\bigg)^{1/2},
}
where $C$ is a constant depending only on $\l_X$ and $\l_{X'}$.
\end{lemma}

\begin{proof}
By properties of product measures, it is clear from the definition \eqref{hellinger_def} that
\eeq{ \label{product_hellinger}
\rho(\l_{(X_1,\dots,X_n)},\l_{(\wt X_1,\dots,\wt X_n)}) = \prod_{i=1}^n \rho(\l_{X_i},\l_{\wt X_i}).
}
Now let $C_0$ be the constant from Lemma \ref{hellinger_lemma}.
From \eqref{tv_hellinger}, \eqref{product_hellinger}, and Lemma \ref{hellinger_lemma}, we deduce
\eq{
d_{\tv}(\l_{(X_1,\dots,X_n)},\l_{(\wt X_1,\dots,\wt X_n)}) \leq \bigg(1 - \prod_{i=1}^n (1-C_0\eps_i^2)^2\bigg)^{1/2}.
}
The desired bound is now obtained by iteratively applying the inequality $(1-x)(1-y) \geq 1-x-y$ for $x,y\geq0$.
\end{proof}

\subsection{Choice of coupling}
Naturally there are many measures $\l_{X'}$ that are absolutely continuous to $\l_X$, but we look for one which can be naturally coupled to $\l_X$ in such a way that $X'$ deviates from $X$ by as much as possible.
Without further assumptions on $\l_X$, the possibilities can be rather limited.
Two choices that are always available, however, are
\eeq{ \label{our_coupling}
X' = \min(X,X^{(1)},\dots,X^{(m)}) \quad \text{or} \quad X' = \max(X,X^{(1)},\dots,X^{(m)}),
}
where $X^{(1)},\dots,X^{(m)}$ are independent copies of $X$.
Indeed, these are the two couplings we will use to prove results on fluctuations in planar random growth models.
It is easy to check that the bounded density condition from Lemma \ref{tv_lemma} is satisfied.

\begin{lemma}
For any law $\l_X$ and any $m\geq1$, the law $\l_{X'}$ of $X'$ given by \eqref{our_coupling} is absolutely continuous with respect to $\l_X$, and has bounded density.
\end{lemma}

\begin{proof}
For any Borel set $A\subset\R$,
\eq{
\P(X'\in A) &\leq\P\bigg(\{X\in A\} \cup \bigcup_{j=1}^m \{X^{(j)}\in A\}\bigg) \\
&\leq \P(X\in A) + \sum_{j=1}^m \P(X^{(j)}\in A)
= (m+1)\P(X\in A).
}
It follows that $\P(X'\in A) = 0$ whenever $\P(X\in A)=0$, and that the density of $\l_{X'}$ with respect to $\l_{X}$ is bounded by $m+1$.
\end{proof}

For a specific distribution $\l_X$, other couplings might also be useful and easier to work with.
For instance, if $X$ is a uniform random variable on $[0,1]$, one could take $X' = aX$ for any $a\in(0,1)$.
If 
$\P(X=0) > 0$, one could simply take $X' = 0$.
For $X$ that is geometrically distributed, $X' = X + a$ is also valid for any positive integer $a$.

\section{Planar random growth models: definitions, background, and results} \label{results}

\subsection{Two-dimensional first-passage percolation} \label{fpp_sec}

Let $ E(\Z^2)$ denote the edge set of $\Z^2$.
Let $(X_e)_{e\in E(\Z^2)}$ be an i.i.d.~family of nonnegative, non-degenerate random variables.
Along a nearest-neighbor path $\gamma = (\gamma_0,\gamma_1,\dots,\gamma_n)$, the \textit{passage time} is 
\eq{
T(\gamma) \coloneqq \sum_{i=1}^n X_{(\gamma_{i-1},\gamma_i)},
}
where $(\gamma_{i-1},\gamma_i)$ denotes the (undirected) edge between $\gamma_{i-1}$ and $\gamma_i$.
For $x,y\in\Z^2$, denote by $T(x,y)$ the minimum passage time of a path connecting $x$ and $y$; that is,
\eq{
T(x,y) \coloneqq \inf\{T(\gamma)\ :\ \gamma_0 = x, \gamma_n = y\}.
}
The quantity $T(x,y)$ is called the \textit{(first) passage time} between $x$ and $y$, and any path achieving this time will be called a \textit{(finite) geodesic}. 
For a recent survey on first-passage percolation, we refer the reader to \cite{auffinger-damron-hanson17}.

We are interested in the fluctuations of $T(x,y)$ when $x$ and $y$ are separated by a distance of order $n$.
In dimensions three and higher, there is actually no known lower bound other than the trivial observation that fluctuations are at least of order $1$.
In the planar setting considered here, order $\sqrt{\log n}$ fluctuations (in the sense of Definition \ref{fluctuations_def}) were established by Pemantle and Peres \cite{pemantle-peres94} when $X_e$ is exponentially distributed.
In \cite[Theorem 2.6]{chatterjee19II}, this lower bound was extended to the family of passage time distributions described in Section \ref{fluctuations_sec}, satisfying \eqref{previous_assumption}.
Our result below expands the result to optimal generality (cf.~Remark \ref{assumption_remark}).

Let $p_c(\Z^d)$ and $\vec{p}_c(\Z^d)$ denote the critical values for undirected and directed bond percolation on $\Z^d$.
When $d = 2$, we have $p_c(\Z^2) = 1/2$ and $\vec{p}_c(\Z^2) \approx 0.6445$ \cite[Chapter 6]{bollobas-riordan06}.
In order to have a rigorous upper bound, we cite the result of \cite{balister-bollobas-stacey94} which guarantees
\eeq{ \label{oriented_bond_upper}
\vec{p}_c(\Z^2) \leq 0.6735.
}

\begin{thm} \label{fpp_thm}
With $s \coloneqq \essinf X_e \in [0,\infty)$, assume
\eeq{
\P(X_e=s) < p_c(\Z^2). \label{fpp_assumption_1}
}
Let $y_n$ be any sequence in $\Z^2$ such that $\|y_n\|_1 \geq n$ for every $n$.
Then the fluctuations of $T(0,y_n)$ are at least of order $\sqrt{\log n}$.
\end{thm}

\begin{remark} \label{assumption_remark}
The above result is optimal in the following sense.
If $s = 0$ and $\P(X_e = 0) > p_c(\Z^d)$, then $T(0,y_n)$ is tight because there is an infinite cluster of zero-weight edges extending in every direction \cite{zhang-zhang84,zhang95}.
\end{remark}

When $s > 0$, we can relax \eqref{fpp_assumption_1} upon adding a weak moment condition \eqref{fpp_assumption_2b}.
This condition is standard in planar FPP and is equivalent to the limit shape having nonempty interior (see \eqref{limit_shape} and the discussion that follows).

\begin{thm} \label{fpp_thm_1}
With $s \coloneqq \essinf X_e \in [0,\infty)$, assume
\begin{subequations} \label{fpp_assumption_2}
\begin{align}
s > 0, \quad \P(X_e=s) < \vec{p}_c(\Z^2), \label{fpp_assumption_2a}
\intertext{and}
\E\min(X^{(1)},X^{(2)},X^{(3)},X^{(4)})^2 < \infty,
\label{fpp_assumption_2b}
\end{align}
\end{subequations}
where the $X^{(i)}$'s are independent copies of $X_e$.
Let $y_n$ be any sequence in $\Z^2$ such that $\|y_n\|_1 \geq n$ for every $n$.
Then the fluctuations of $T(0,y_n)$ are at least of order $\sqrt{\log n}$.
\end{thm}

\begin{remark}
As similarly mentioned in Remark \ref{assumption_remark}, the above result is optimal in the following sense.
If $s > 0$ and $\P(X_e = s) > \vec{p}_c(\Z^d)$, then $T(0,y_n) - n\|y_n\|_1$ is tight so long as $y_n$ is in or at the edge of the oriented percolation cone \cite[Remark 7]{zhang08} (c.f.~\cite{durrett84} for a description of this cone).
An independent work of Damron, Hanson, Houdr{\'e}, and Xu \cite{damron-hanson-houdre-xu20}, which uses different methods and was posted shortly after a first version of this manuscript, shows that Theorem \ref{fpp_thm_1} holds even if one assumes \eqref{fpp_assumption_2a} without \eqref{fpp_assumption_2b};  their Lemma 6 is the key innovation needed to remove this moment condition.
They also prove a statement equivalent to Theorem \ref{fpp_thm}.
\end{remark}

One should compare Theorems  \ref{fpp_thm} and \ref{fpp_thm_1} with the results of Newman and Piza \cite{newman-piza95}.
Under \eqref{fpp_assumption_1} or \eqref{fpp_assumption_2a}, and the additional assumption that $\E(X_e^2)$ is finite --- which is slightly stronger than \eqref{fpp_assumption_2b} --- they show $\Var(T(0,y_n))\geq C\log n$.
Zhang \cite[Theorem 2]{zhang08} shows the same for $y_n = (n,0)$ assuming only $\P(X_e = 0) < p_c(\Z^2)$, and Auffinger and Damron \cite[Corollary 2]{auffinger-damron13} extend this result to any direction outside the percolation cone (see also \cite[Corollary 1.3]{kubota15}).
Unfortunately, these lower bounds on variance give no information on the true size of fluctuations, hence the need for Theorems \ref{fpp_thm} and \ref{fpp_thm_1}.
Indeed, one cannot expect a matching upper bound since $\Var(T(0,y_n))$ should be of order $n^{2/3}$ in the standard cases.

The best known variance upper bound is $Cn/\log n$, proved in general dimensions for progressively more general distributions by Benjamini, Kalai, and Schramm \cite{benjamini-kalai-schramm03}, Bena\"\i m and Rossignol \cite{benaim-rossignol08}, and Damron, Hanson, and Sosoe \cite{damron-hanson-sosoe15,damron-hanson-sosoe14}.
One notable exception to the $n/\log n$ barrier comes from a simplified FPP model introduced by Sepp{\"a}l{\"a}inen \cite{seppalainen98}, for which Johansson \cite[Theorem 5.3]{johansson01} proves that the passage time fluctuations, when rescaled by a suitable factor of $n^{1/3}$, converge to the GUE Tracy--Widom distribution \cite{tracy-widom94}. 

Interestingly, in the critical case $\P(X_e=0)=1/2$ with $\P(0<X_e<\eps) = 0$, fluctuations are of order exactly $\sqrt{\log n}$.
Kesten and Zhang \cite{kesten-zhang97} prove a central limit theorem on this scale, and in the binary case $\P(X_e=1) = 1/2$, Chayes, Chayes, and Durrett \cite[Theorem 3.3]{chayes-chayes-durrett86} establish the expected asymptotic $\E(T(0,n\mathrm{e}_1)) = \Theta(\log n)$.
More delicate critical cases are examined in \cite{zhang99,damron-lam-wang17}.

Next we turn our attention to the related shape fluctuations.
For $x\in\R^2$, let $[x]$ be the unique element of $\Z^2$ such that $x\in[x]+[0,1)^2$.
For each $t>0$, define
\eeq{ \label{limit_shape}
B(t) \coloneqq \{x \in \R^2 : T(0,[x]) \leq t\},
}
which encodes the set of points reachable by a path of length at most $t$.
Sharpened from a result of Richardson \cite{richardson73}, the Cox--Durrett shape theorem  \cite[Theorem 3]{cox-durrett81} says that if (and only if) $\P(X_e = 0) < p_c(\Z^2)$ and \eqref{fpp_assumption_2b} holds,
then there
exists a deterministic, convex, compact set $\b\subset\R^2$, having the symmetries of $\Z^2$ and nonempty interior, such that for any $\eps > 0$, almost surely
\eq{
(1 - \eps)\b \subset \frac{1}{t}B(t) \subset (1 + \eps)\b \quad \text{for all large $t$}.
}
More specifically, for every $x\in\R^2$, there is a positive, finite constant $\mu(x)$ such that
\eeq{ \label{mu_def}
\lim_{n\to\infty} \frac{T(0,[nx])}{n} = \mu(x) \quad \mathrm{a.s.},
}
and
\eq{
\b = \{x \in \R^2 : \mu(x) \leq 1\}.
}
Moreover, $\mu$ is a norm on $\R^2$, and so $\b$ is the unit ball under this norm.

The question remains as to how far $B(t)$ typically is from $t \b$.
One way to pose this problem precisely is to ask for the value of
\eeq{ \label{chi_prime_def}
\chi' \coloneqq \inf\Big\{\nu : \P\big((t-t^\nu)\b \subset B(t) \subset (t+t^\nu)\b \text{ for all large $t$}\big) = 1\Big\}.
}
Another possible quantity to consider is $\chi \coloneqq \sup_{\|x\|_2=1} \chi_x$, where
\eq{
\chi_x \coloneqq \sup\{\gamma\geq0 : \exists\, C>0,\Var T(0,[nx]) \geq Cn^{2\gamma} \text{ for all $n$}\}.
}
Although it is conjectured that $\chi_x = \chi = \chi' = \frac{1}{3}$, even relating $\chi$ and $\chi'$ is  challenging because a variance lower bound does not by itself guarantee anything about fluctuations.
Assuming $\E(X_e^2) < \infty$ and either \eqref{fpp_assumption_1} or \eqref{fpp_assumption_2a}, Newman and Piza \cite[Theorem 7]{newman-piza95} prove $\max(\chi,\chi') \geq 1/5$.
Furthermore, they show $\chi_x \geq 1/8$ if $x$ is a direction of curvature for $\b$, a notion defined in \cite{newman-piza95} and recalled here.

\begin{defn} \label{curvature_def}
Let $x\in\R^2$ be a unit vector, and $z\in\partial \b$ the boundary point of $\b$ in the direction $x$.
We say $x$ is a direction of curvature for $\b$ if there exists a Euclidean ball $\SS$ (with any center and positive radius) such that $\SS\supset\b$ and $z\in\partial\SS$.
\end{defn}

Since $\b$ must have at least one direction of curvature (e.g.~take a large ball $\SS$ containing $\b$, and then translate $\SS$ until it first intersects $\partial\b$), one has $\chi \geq 1/8$ in the setting of \cite{newman-piza95}.
Unfortunately, this result does not imply order $n^{1/8}$ fluctuations without a matching upper bound on the variance.

The first work addressing typical shape fluctuations is due to Zhang \cite{zhang06}, who shows they are at least of order $\sqrt{\log n}$ in a certain sense for Bernoulli weights and general dimension.
Nakajima \cite{nakajima20} extends this result to general distributions.
In the first result proving $\chi' > 0$, Chatterjee \cite[Theorem 2.8]{chatterjee19II} shows that if for some direction of curvature $x$, $T(0,[nx])$ has fluctuations of order $n^{1/8-\delta}$ for any $\delta > 0$ in the sense of Definition \ref{fluctuations_def}, then $\chi' \geq 1/8$.
It is then shown in \cite[Theorem 2.7]{chatterjee19II} that the hypothesis of the previous sentence is true if the weight distribution satisfies \eqref{previous_assumption}.
Here we are able to replace that assumption with a small moment condition needed to use Alexander's shape theorem \cite{alexander97}, as refined by Damron and Kubota \cite{damron-kubota16}.

\begin{thm} \label{fpp_thm_2}
Assume $\P(X_e=0) < p_c(\Z^2)$ and $\E(X_e^\lambda)<\infty$ for some $\lambda > 3/2$.
If $x$ is a direction of curvature for $\b$, then $T(0,[nx])$ has fluctuations of order at least $n^{1/8-\delta}$ for any $\delta > 0$.
\end{thm}

By the argument of \cite[Theorem 2.8]{chatterjee19II}, we obtain the following lower bound on the shape fluctuation exponent.

\begin{cor}
Assume the setting of Theorem \ref{fpp_thm_2}.
Then the shape fluctuation exponent defined by \eqref{chi_prime_def} satisfies $\chi' \geq \frac{1}{8}$.
\end{cor}

\subsection{Corner growth model} \label{lpp_sec}

In its planar form, LPP is often called the \textit{corner growth model}.
It is similar to FPP, the main differences being that only directed paths are considered (i.e.~coordinates never decrease), and the passage time $T$ is defined by time-maximizing paths rather than minimizing ones.
Furthermore, by convention we place the weights on the vertices instead of the edges, but this difference is more technical than conceptual.
We will now make this setup precise.

Let $\Z^2_+$ denote the first quadrant of the square lattice, that is the set of all $v = (a,b) \in \Z^2$ with $a,b\geq 0$.
We will write the standard basis vectors as $\mathbf{e}_1 = (1,0)$ and $\mathbf{e}_2 = (0,1)$.
Let $(X_v)_{v\in\Z^2_+}$ be an i.i.d.~family of non-degenerate random variables; because of the directedness, no assumption of nonnegativity is needed.
A \textit{directed} path $\vec\gamma = (\gamma_0,\gamma_1,\dots,\gamma_n)$ is one in which each increment $\gamma_i - \gamma_{i-1}$ is equal to $\mathbf{e}_1$ or $\mathbf{e}_2$.
The \textit{passage time} of such a path is
\eq{
T(\vec\gamma) \coloneqq \sum_{i=1}^n X_{\gamma_i}.
}
Let $T(u,v)$ be the maximum passage time of a directed path from $u$ to $v$, called the \textit{(last) passage time},
\eq{
T(u,v) \coloneqq \sup\{T(\vec\gamma)\ |\ \gamma_0 = u, \gamma_n = v\}.
}
We will again refer to any path achieving this time as a \textit{(finite) geodesic}.
Once more $T$ satisfies a shape theorem under mild assumptions on $\l_{X}$, which we will not discuss. 
For further background, the reader is directed to \cite{martin06,quastel-remenik14,rassoul18}.

The directed structure advantages this model because of correspondences with problems in queueing networks, interacting particle systems, combinatorics, and random matrices.
Remarkable progress has been made by leveraging these connections in specific cases, leading to rigorous proofs of order $n^{1/3}$ passage time fluctuations converging to Tracy--Widom distributions upon rescaling.
This has been successfully carried out by Johansson \cite{johansson00} when the $X_v$'s are geometrically or exponentially distributed, building on work of Baik, Deift, and Johansson \cite{baik-deift-johansson99} connected to a continuum version of LPP.
The results extend to point-to-line passage times \cite{borodin-ferrari-prahofer-sasamoto07}.
Purely probabilistic techniques for accessing fluctuation exponents appear in \cite{cator-groeneboom06, balazs-cator-seppalainen06}.
The fluctuation exponent of $1/3$ is also present in a model known as Brownian LPP, for which the connection to Tracy--Widom laws is more explicit \cite{oconnell03}.

Away from exactly solvable settings, Chatterjee \cite[Theorem 8.1]{chatterjee08} proves that when the vertex weights are Gaussian, the point-to-line passage time has variance at most $Cn/\log n$.
Graham \cite{graham12} extends this result to general dimensions, also discussing uniform and gamma distributions.
To our knowledge, no general lower bound on fluctuations has been written for LPP.
It is worth mentioning, however, that the results in \cite{newman-piza95} are also stated for \textit{directed} FPP.
It is natural to suspect that many of results mentioned for FPP could be naturally translated to the LPP setting.
Indeed, as we now discuss, Theorem \ref{fpp_thm} carries over with little modification.

Let $\vec{p}_{c,\,\mathrm{site}}(\Z^2)$ be the critical value of directed site percolation on $\Z^2$.
It is clear that $\vec{p}_{c,\,\mathrm{site}}(\Z^2)$ is at least as large as its undirected counterpart ${p}_{c,\,\mathrm{site}}(\Z^2)$, which in turn satisfies ${p}_{c,\,\mathrm{site}}(\Z^2) > p_c(\Z^2) = 1/2$ \cite{grimmett-stacey98}.
In the way of upper bounds, it is known from \cite{balister-bollobas-stacey94,liggett95} that $\vec{p}_{c,\,\mathrm{site}}(\Z^2) \leq 3/4$. 
Let $S \coloneqq \esssup X_v \in (-\infty,\infty]$.
The assumption analogous to \eqref{fpp_assumption_1} or \eqref{fpp_assumption_2a} is
\eeq{ \label{lpp_assumption}
\P(X_v = S) < \vec{p}_{c,\,\mathrm{site}}(\Z^2).
}

\begin{thm} \label{lpp_thm}
Assume \eqref{lpp_assumption}.
Let $v_n$ be any sequence in $\Z^2_+$ such that $\|v_n\|_1 \geq n$ for every $n$.
Then the fluctuations of $T(0,v_n)$ are at least of order $\sqrt{\log n}$.
\end{thm}

In the case $v_n = n\mathbf{e}_1$, the passage time $T(0,n\mathbf{e}_1)$ is just the sum of $n$ i.i.d.~random variables and thus fluctuates on the scale of $n^{1/2}$.
The $n^{1/3}$ scaling should manifest when the two coordinates of $v_n$ are both of order $n$.
Interpolating between these two regimes, it is expected that if $v_n = (n,\floor{n^{a}})$ for $a \in (0,1)$, then $T(0,v_n)$ has fluctuations of order $n^{1/2-a/6}$. 
Such a result is proved, along with rescaled convergence to the GUE Tracy--Widom distribution, for $a < 3/7$ \cite{baik-suidan05,bodineau-martin05}.

\subsection{Directed polymers in $1+1$ dimensions} \label{dp_sec}

The model of directed polymers in random environment is a positive-temperature version of LPP.
That is, instead of examining only maximal paths, we consider the softer model of defining a Gibbs measure on paths, with those of greater passage time receiving a higher probability.
With $\Z^2_+$ as before, we again take $(X_v)_{v\in\Z^2_+}$ to be an i.i.d.~family of non-degenerate random variables, called the \textit{random environment}.
Let $\vec\Gamma_n$ denote the set of directed paths $\vec\gamma = (v_0,v_1,\dots,v_n)$ of length $n$ starting at the origin $v_0 = 0$.
Given an inverse temperature $\beta > 0$, define a Gibbs measure $\rho_n^\beta$ on $\vec\Gamma_n$ by
\eq{
\rho_n^\beta(\vec\gamma) \coloneqq \frac{\e^{\beta H_n(\vec\gamma)}}{Z_n^\beta}, \qquad H_n(\vec\gamma) \coloneqq \sum_{i=1}^n X_{v_i}, \quad \vec\gamma \in \vec\Gamma_n, 
}
where now the object of interest is the \textit{partition function},
\eq{
Z_n^\beta \coloneqq \sum_{\vec\gamma\in\vec\Gamma_n} \e^{\beta H_n(\vec\gamma)}.
}
Since $Z_n^\beta$ grows exponentially in $n$, the proper linear quantity to consider is the \textit{free energy}, $\log Z_n^\beta$.
Strictly speaking, the following result is not the exact analogue of Theorems \ref{fpp_thm} and \ref{lpp_thm}, since we have not fixed the endpoint.
Nevertheless, the same argument goes through for point-to-point free energies.

\begin{thm} \label{dp_thm}
Assume \eqref{lpp_assumption}.
Then the fluctuations of $\log Z_n^\beta$ are at least of order $\sqrt{\log n}$ for any $\beta > 0$.
\end{thm}

As in LPP, there are several exactly solvable models of $(1+1)$-dimensional directed polymers for which free energy fluctuations on the order of $n^{1/3}$ can be calculated, beginning with the inverse-gamma (or log-gamma) polymer introduced by Sepp{\"a}l{\"a}inen \cite{seppalainen12}.
There are now three other solvable models: the strict-weak polymer \cite{corwin-seppalainen-shen15,oconnell-ortmann15}, the Beta RWRE \cite{barraquand-corwin17}, and the inverse-beta polymer \cite{thiery-doussal15}.
Chaumont and Noack show in \cite{chaumont-noack18I} that these are the only possible models possessing a certain stationarity property, and in \cite{chaumont-noack18II} provide a unified approach to calculating their fluctuation exponents.
We also mention the positive temperature version of Brownian LPP, introduced by O'Connell and Yor \cite{oconnell-yor01}, for which order $n^{1/3}$ energy fluctuations have been established \cite{seppalainen-valko10,borodin-corwin14,borodin-corwin-ferrari14}.

For the general model considered here, the situation is much the same as for FPP.
In the way of upper bounds, Alexander and Zygouras \cite{alexander-zygouras13} prove exponential concentration of $\log Z_n^\beta - \E(\log Z_n^\beta)$ on the scale of $\sqrt{n/\log n}$, in analogy with works mentioned earlier \cite{benjamini-kalai-schramm03,benaim-rossignol08,damron-hanson-sosoe15,damron-hanson-sosoe14,chatterjee08,graham12}.
Their results hold in general dimensions and for a wide range of distributions.
As for lower bounds, Piza \cite{piza97} proves $\Var(\log Z_n^\beta) \geq C\log n$ for non-positive weights with finite variance, as well as weaker versions of the shape theorem results from \cite{newman-piza95}.

Although Theorem \ref{dp_thm} does not even prove a positive fluctuation exponent, simply knowing that free energy fluctuations diverge may be significant in understanding the phenomenon of polymer localization.
One way of defining this phenomenon is to say the polymer measure is \textit{localized} if its endpoint distribution has atoms:
\eeq{ \label{localization}
\limsup_{n\to\infty} \max_{\|v\|_1=n} \rho_n^\beta(\gamma_n = v) > 0 \quad \mathrm{a.s.}
}
It is known \cite[Proposition 2.4]{comets-shiga-yoshida03} that \eqref{localization} occurs for any $\beta > 0$ in $1+1$ and $1+2$ dimensions, and for sufficiently large $\beta$ in higher dimensions, depending on the law of the $X_v$'s.
What is unclear, however, is whether the atoms or ``favorite endpoints" are typically close to one another or far apart.
From the solvable case \cite{seppalainen12}, there is evidence suggesting the former is true at least in $1+1$ dimensions \cite{comets-nguyen16}.
In general dimensions, the same is known only along random subsequences \cite{bates18,bates-chatterjee20}.
These subsequences also exist for polymers on trees, but in that setting, the favorite sites more frequently appear far apart \cite{barral-rhodes-vargas12}; this behavior is thus difficult to rule out in high-dimensional lattices.
It is interesting, then, that for both polymers on trees and for high-temperature lattice polymers in dimensions $1+3$ and higher, the fluctuations of $\log Z_n^\beta$ are order $1$.
On the lattice, this fact is easy to deduce from a martingale argument; see \cite[Chapter 5]{comets17}.
For the tree case, see \cite[Section 5]{derrida-spohn88}.

\section{Proofs of main results} \label{proofs}

The proofs follow a general strategy, which we outline below.
For clarity, we will break each proof into two parts:
\begin{itemize}[leftmargin=0.75in]
\item[{\textbf{Part 1.}}] Use the coupling \eqref{our_coupling} with large enough $m$ to show that in all relevant paths, there is a high frequency of weights where $X'$ is far away from $X$.
\item[{\textbf{Part 2.}}] Show the same is true when $X'$ is replaced by $\wt X$ defined by \eqref{new_X}, provided we make good choices for $\eps$.
This step uses Part 1, as well as the independence of $Y$ from $X$ and $X'$.
Conclude that the passage time (or free energy) has, with positive probability independent of $n$, changed by an amount of the desired order.
\end{itemize}

\subsection{Proof of Theorems \ref{fpp_thm} and \ref{fpp_thm_1}}

Recall the notation
\eq{
s = \essinf X_e.
}
Before proceeding with the main argument, we begin with a lemma meant to guarantee that geodesics contain many edges with weights far from $s$.
Preempting a technical concern, we note that with probability $1$, geodesics do exist between all pairs of points in $\Z^2$ without any assumptions on the distribution of $X_e$ \cite{wierman-reh78}.
We will use the notation $B_n(x) \coloneqq \{y\in\Z^2 : \|x-y\|_1\leq n$\} and $\partial B_n(x) \coloneqq \{y \in \Z^2 : \|x-y\|_1 = n\}$ for $n\geq1$.

\begin{lemma} \label{fpp_lemma}
Given $\delta>0$ and $\rho \in (0,1)$, let $E_n^x$ be the event that there exists a geodesic $\gamma = (\gamma_0,\gamma_1,\dots,\gamma_N)$ from $x\in\Z^2$ to some $y\in \partial B_n(x)$ such that 
\eeq{ \label{bad_proportion}
\#\{1 \leq i \leq N : X_{(\gamma_{i-1},\gamma_i)} \geq s + 2\delta\} < \rho n.
}
If \eqref{fpp_assumption_1} or \eqref{fpp_assumption_2} holds, then there are $\delta$ and $\rho$ sufficiently small that
\eeq{ \label{sum_n_bd_1}
\sum_{n = 1}^\infty \P(E_n^0) < \infty.
}
Furthermore, for some sequence $(n_k)_{k=1}^\infty$ satisfying $2^{k-1}< n_k\leq 2^{k}$,
\eeq{ \label{sum_n_bd_2}
\sum_{k=1}^\infty\sum_{\|x\|_1 = n_k} \P(E_{n_k}^x) < \infty.
}
\end{lemma}

\begin{remark}
As will be seen in the proof, the restriction of Lemma \ref{fpp_lemma} to geodesics is only necessary when assuming \eqref{fpp_assumption_2} without \eqref{fpp_assumption_1}.
\end{remark}

We will need two results from the literature.
The first theorem below was originally established by van den Berg and Kesten \cite{vandenberg-kesten93} when $y = (1,0)$, and later generalized by Marchand \cite{marchand02}.

\begin{thm}[{Marchand \cite[Theorem 1.5(ii)]{marchand02}}] \label{marchand_thm}
Let $(X_e)_{e\in E(\Z^2)}$ and $(\hat X_e)_{e\in E(\Z^2)}$ be two i.i.d.~families of nonnegative random variables, such that $\hat X_e$ stochastically dominates $X_e$.
Let $\mu$ and $\hat \mu$ be the respective limiting norms, given by \eqref{mu_def}.
If $\P(X_e=s) < \vec{p}_c(\Z^2)$, then $\mu(y) < \hat \mu(y)$ for all $y \neq 0$.
\end{thm}

The next theorem demonstrates why \eqref{fpp_assumption_2b} is necessary when \eqref{fpp_assumption_1} is not assumed.
The version stated in \cite{ahlberg15} uses $\|\cdot\|_2$ in place of $\|\cdot\|_1$, but this makes no difference because all norms on $\R^2$ are equivalent.

\begin{thm}[{Ahlberg \cite[Theorem 1]{ahlberg15}}] \label{ahlberg_thm}
For every $\alpha, \eps>0$,
\eq{
\E\min(X^{(1)},X^{(2)},X^{(3)},X^{(4)})^\alpha<\infty \quad \iff \quad
\sum_{y\in\Z^2}\|y\|_1^{\alpha-2}\P(|T(0,y) - \mu(y)| > \eps\|y\|_1) < \infty,
}
where the $X^{(i)}$'s are independent copies of $X_e$.

\end{thm}

\begin{proof}[Proof of Lemma \ref{fpp_lemma}]
We handle the cases of \eqref{fpp_assumption_1} and \eqref{fpp_assumption_2} separately.

\textbf{Case 1: Assuming \eqref{fpp_assumption_1}.}
Choose $\delta > 0$ small enough that $\P(X_e < s+ 2\delta) < p_c(\Z^2) = 1/2$.
Consider the first-passage percolation when each $X_e$ is replaced by
\eq{
\hat X_e \coloneqq 
\begin{cases}
0 &\text{if $X_{e} < s+2\delta$,} \\
1 &\text{ otherwise.}
\end{cases}
}
Let $\hat T$ be the associated passage time, so that $\hat T(x,y)$ is simply the minimum number of edges $e$ satisfying $X_e \geq s+2\delta$ in a path from $x$ to $y$.
By \cite[Theorem 1]{kesten80}, there exists $\rho$ small enough that with probability tending to $1$ exponentially quickly in $n$, every self-avoiding path $\gamma$ starting at the origin that has length at least $n$ --- not just those terminating at $\partial B_n(0)$ --- has $\hat T(\gamma) \geq \rho n$.
That is, $\P(E_n^0) \leq a\e^{-bn}$ for some $a,b>0$, which easily gives
\eq{
\sum_{n=1}^\infty n\P(E_n^0) < \infty.
}
In particular, \eqref{sum_n_bd_1} is true, and  \eqref{sum_n_bd_2} holds for any increasing sequence $n_k\to\infty$, since $|\partial B_n(0)| = 4n$ for every $n\geq1$.

\textbf{Case 2: Assuming \eqref{fpp_assumption_2}.}
Recall that \eqref{fpp_assumption_2b} implies the existence of the finite limit \eqref{mu_def} for every $x\in\R^2$.
By \eqref{fpp_assumption_2a}, we can choose $\delta > 0$ small enough that $\P(X_e < s+2\delta) < \vec{p}_c(\Z^2)$.
Next we choose $M$ large enough that $\P(s+2\delta \leq X_e < s+2\delta+M) \geq 1/4$, which is possible because of \eqref{oriented_bond_upper}.
Consider the first-passage percolation model where each $X_e$ is replaced by
\eq{
\hat X_e \coloneqq \begin{cases}
s+2\delta + M &\text{if } s+2\delta\leq X_e< s+2\delta+M, \\
X_e &\text{otherwise.}
\end{cases}
}
Let $\hat T$ and $\hat \mu$ be the associated passage time and limiting norm.
We also define
\eq{
\mu_\mathrm{min} &\coloneqq \min\{\mu(y) : y\in\R^2, \|y\|_1=1\},
}
which is positive because $s>0$, 
and finite because of \eqref{fpp_assumption_2b}.
Because of our choice of $\delta$ and $M$, Theorem \ref{marchand_thm}
guarantees $\mu(y) < \hat \mu(y)$ for every nonzero $y \in \R^2$.
By compactness and continuity of $\mu$ and $\hat\mu$,
there is $\eps_1 > 0$ such that $\mu(y)(1 + \eps_1) < \hat \mu(y)(1-2\eps_1)$ for every $y$ with $\|y\|_1 = 1$.
By scaling, the same inequality holds for all $y\neq 0$.
Therefore, if we set $\eps_2 \coloneqq \eps_1\min(\mu_\mathrm{min},1)$,
then for all $y\in \partial B_n(0)$,
\eeq{ \label{scaling_ineq}
\mu(y)(1+\eps_1) + \eps_2 n
&< \hat\mu(y)(1-2\eps_1) + \eps_2 \|y\|_1 \\
&\leq \hat\mu(y)(1-2\eps_1) + \eps_1\mu(y) \\
&< \hat\mu(y)(1-2\eps_1) + \eps_1\hat \mu(y)
= \hat\mu(y)(1-\eps_1).
}
Finally, choose $\rho \in (0,1)$ such that $\rho M < \eps_2$.

Now consider any $y\in \partial B_n(0)$.
If there exists a geodesic $\gamma$ (with respect to $T$) from $0$ to $y$ such that \eqref{bad_proportion} holds, then $\gamma$ contains fewer than $\rho n$ edges $e$ such that $\hat X_e \neq X_e$.
Moreover, for each such edge, we have $\hat X_e\leq X_e + M$.
Therefore,
\eq{
\hat T(0,y) \leq \sum_{i=1}^N \hat X_{(\gamma_{i-1},\gamma_i)} \leq T(0,y) + n\rho M < T(0,y)+\eps_2 n.
}
But in light of \eqref{scaling_ineq},
\eq{
\{T(0,y) \leq \mu(y)(1+\eps_1)\}\cap\{\hat\mu(y)(1-\eps_1) \leq \hat T(0,y)\} \subset \{T(0,y) + \eps_2n \leq \hat T(0,y)\}.
}
From these observations, we see
\eeq{ \label{complement_containment}
E_n^0 \subset
\bigcup_{\|y\|_1=n} \{T(0,y) > \mu(y)(1+\eps_1)\}\cup\{\hat T(0,y) < \hat\mu(y)(1-\eps_1)\},
}
and hence 
\eq{
\P(E_n^0) &\leq \sum_{\|y\|_1=n} \big[\P\big(T(0,y) - \mu(y) > \eps_1\mu(y)\big) + \P\big(\hat T(0,y) - \hat\mu(y) < -\eps_1\hat\mu(y)\big)\big] \\
&\leq \sum_{\|y\|_1=n} \big[\P\big(T(0,y) - \mu(y) > \eps_2\|y\|_1\big) + \P\big(\hat T(0,y) - \hat\mu(y) < -\eps_2\|y\|_1\big)\big].
}
By Theorem \ref{ahlberg_thm} with $\alpha=2$, 
\eqref{fpp_assumption_2b} gives
\eq{
\sum_{y\in\Z^2} \big[\P\big(T(0,y) - \mu(y) > \eps_2\|y\|_1\big) + \P\big(\hat T(0,y) - \hat\mu(y) < -\eps_2\|y\|_1\big)\big] < \infty.
}
Now \eqref{sum_n_bd_1} follows from the previous two displays.
To conclude \eqref{sum_n_bd_2}, we take
\eq{
n_k \coloneqq \argmin_{2^{k-1}<n\leq2^k} \P(E_n^0).
}
Note that by translation invariance, $\P(E_n^x) = \P(E_n^0)$ for all $x\in\Z^2$.
Again using the fact that $|\partial B_n(0)| = 4n$ for all $n\geq1$, we have 
\eq{
\sum_{k=1}^\infty\sum_{\|x\|_1 = n_k} \P(E_{n_k}^x) = 4\sum_{k=1}^\infty n_k\P(E_{n_k}^0)
\leq 8\sum_{k=1}^\infty2^{k-1}\P(E_{n_k}^0)
&\leq 8\sum_{k=1}^\infty \sum_{n=2^{k-1}+1}^{2^{k}} \P(E_n^0) \\
&= 8\sum_{n=2}^\infty \P(E_n^0) < \infty.
}
\end{proof}

\begin{proof}[Proof of Theorems \ref{fpp_thm} and \ref{fpp_thm_1}] \hspace{1in}\\[-0.5\baselineskip]

\textbf{Part 1.} Let $T_n = T(0,y_n)$.
From Lemma \ref{fpp_lemma}, take $\delta > 0$, $\rho \in (0,1)$, and $(n_k)_{k=1}^\infty$ satisfying $2^{k-1}<n_k\leq2^k$, such that \eqref{sum_n_bd_2} holds.
Then choose $k_0$ large enough that
\eq{
\sum_{k=k_0}^{\infty} \sum_{\|x\|_1=n_k} \P(E_{n_k}^x) \leq \frac{1}{7}.
}
Define the event 
\eeq{ \label{G_0_def}
G_0 \coloneqq \bigcap_{k=k_0}^{\infty} \bigcap_{\|x\|_1=n_k} (E^x_{n_k})^\mathrm{c},
}
so that
\eeq{
\P(G_0) \geq \frac{6}{7}. \label{G_bd}
}
Finally, choose $m$ large enough that if $X_e^{(1)},\dots,X_e^{(m)}$ are independent copies of $X_e$, then 
\eeq{
\P(\min(X_e^{(1)},\dots,X_e^{(m)})\leq s+\delta) > 1 - \Big(\frac{1}{3}\Big)^{1/\rho}. \label{m_choice}
}
Throughout the rest of the proof, $C$ will denote a constant that may depend on $m$ and $\l_{X}$, but nothing else.
Its value may change from line to line or within the same line.
To condense notation, we will also define 
\eeq{ \label{prime_definitions}
X_e' \coloneqq \min(X_e,X_e^{(1)},\dots,X_e^{(m)}), \qquad Z_e \coloneqq X_e - X_e', \qquad W_e \coloneqq 1 - \e^{-Z_e},
}
where $(X_e^{(j)})_{e\in E(\Z^2)}$, $1\leq j\leq m$, are independent copies of the i.i.d.~edge weights.

Given any realization of the percolation, the subgraph of $\Z^2$ induced by the geodesics between all pairs of points in $B_{2n}(0)$ is finite and connected.
Therefore, we can choose one of its spanning trees according to some arbitrary, deterministic rule.
From that tree we have a distinguished geodesic for each $x,y\in B_{2n}(0)$.
Moreover, if $x'$ and $y'$ lie along the geodesic from $x$ to $y$, then the distinguished geodesic from $x'$ to $y'$ is the relevant subpath. 

Given $c>0$ to be chosen later, consider the event $F_n$ that there exist $x\in  \partial B_n(0)$ and $y\in\partial B_{2n}(0)$ whose distinguished geodesic --- which we denote by its edges $(e_1,\dots,e_N)$ in a slight abuse of notation --- satisfies
\eeq{ \label{defining_F}
\sum_{i=1}^N W_{e_i} \leq cn.
}
For a given $x\in B_n(0)$, if $E_n^x$ does not occur, then any geodesic from $x$ to any $y\in\partial B_n(x)$ contains at least $\rho n$ edges satisfying $X_{e_i} \geq s + 2\delta$.
Furthermore, because $\|x\|_1=n$, every geodesic from $x$ to $\partial B_{2n}(0)$ must pass through $\partial B_n(x)$.
Therefore, if $E_n^x$ does not occur, then any geodesic from $x$ to $\partial B_{2n}(0)$ contains at least $\rho n$ edges satisfying $X_{e_i} \geq s + 2\delta$.

It will be convenient to define
\eq{
U_{e} \coloneqq \begin{cases}
1 &\text{if } \min(X_e^{(1)},\dots,X_e^{(m)}) \leq s + \delta, \\
0 &\text{otherwise.}
\end{cases}
}
The reason for doing so is that now the $U_e$'s are mutually independent and independent of $\sigma(X)$, the $\sigma$-algebra generated by the $X_e$'s.
In addition, if $(e_1,\dots,e_N)$ is the distinguished geodesic between some fixed $x\in\partial B_n(0)$ and $y\in\partial B_{2n}(0)$, then from the observation
\eq{
X_{e}\geq s+2\delta,\ U_{e} = 1 \quad \implies \quad Z_{e} \geq \delta \quad \implies \quad W_{e} \geq 1-\e^{-\delta},
}
we see
\eq{
\sum_{i=1}^N \one_{\{X_{e_i}\geq s+2\delta\}}\one_{\{U_{e_i}=1\}}(1-\e^{-\delta}) \leq \sum_{i=1}^N W_{e_i}.
}
By the discussion of the previous paragraph, if $E_n^x$ does not occur, then there is a subsequence $1\leq i_1<i_2<\cdots<i_{\ceil{\rho n}}\leq N$ such that $X_{e_{i_\ell}}\geq s+2\delta$ for each $\ell=1,\dots,\ceil{\rho n}$.
With this notation, we have
\eq{
\P\givenp[\bigg]{\sum_{i=1}^N W_{e_i} \leq cn}{\sigma(X)}\one_{(E_n^x)^\mathrm{c}}
&\leq \P\givenp[\bigg]{\sum_{i=1}^N \one_{\{X_{e_i}\geq s+2\delta\}}\one_{\{U_{e_i}=1\}}(1-\e^{-\delta})\leq cn}{\sigma(X)}\one_{(E_n^x)^\mathrm{c}} \\
&\leq \P\givenp[\bigg]{\sum_{\ell=1}^{\ceil{\rho n}} \one_{\{U_{e_{i_\ell}}=1\}}(1-\e^{-\delta})\leq cn}{\sigma(X)}\one_{(E_n^x)^\mathrm{c}} \\
&\leq \phi(t)^{\ceil{\rho n}}\exp\Big\{\frac{cnt}{1-\e^{-\delta}}\Big\} \quad \text{for any $t > 0$,}
}
where
\eq{
\phi(t) \coloneqq \E(\e^{-t U_e}) = \P(U_e=0) + \e^{-t}\P(U_e=1) 
\stackrel{\mbox{\footnotesize\eqref{m_choice}}}{<} \Big(\frac{1}{3}\Big)^{1/\rho} +  \e^{-t}\Big(1-\frac{1}{3^{1/\rho}}\Big). 
}
We can choose $t$ sufficiently large that $\phi(t)^\rho \leq 1/3$.
Then setting $c = (1-\e^{-\delta})t^{-1}$, we have 
\eq{
\P\givenp[\bigg]{\sum_{i=1}^N W_{e_i} \leq cn}{\sigma(X)}\one_{(E_n^x)^\mathrm{c}} \leq \frac{\e^n}{3^n}.
}
We now use this estimate to bound the conditional probability of the event $F_n$ defined via \eqref{defining_F}.
Since $|\partial B_n(0)| = 4n$ and $|\partial B_{2n}(0)| = 8n$, a union bound gives
\eeq{ \label{conditional_F_bd}
\P\givenp{F_n}{\sigma(X)}\one_{\{\bigcap_{\|x\|_1= n}(E_n^x)^\mathrm{c}\}} \leq \frac{32n^2\e^n}{3^n} \quad \text{for all $n\geq1$.}
}
Now we choose an even integer $k_1\geq k_0$ sufficiently large that
\eeq{ \label{sum_with_n_k}
32\sum_{k=k_1}^\infty \frac{n_k^2\e^{n_k}}{3^{n_k}} \leq \frac{1}{8},
}
and define the event
\eeq{ \label{G_def}
G \coloneqq \bigcap_{k=k_1}^{\infty} F_{n_k}^\mathrm{c}.
} 
Recall the event $G_0 \in \sigma(X)$ defined in \eqref{G_0_def}.
The above discussion yields
\eq{ 
\P\givenp{G}{\sigma(X)}\one_{G_0} 
&\stackrel{\phantom{{\mbox{\footnotesize{\eqref{conditional_F_bd}}}}}}{\geq} \bigg(1-\sum_{k=k_1}^{\infty} \P\givenp{F_{n_k}}{\sigma(X)} \bigg)\one_{G_0} \\
&\stackrel{\phantom{{\mbox{\footnotesize{\eqref{conditional_F_bd}}}}}}{=} \bigg(1-\sum_{k=k_1}^{\infty} \P\givenp{F_{n_k}}{\sigma(X)} \bigg)\prod_{k=k_0}\one_{\{\bigcap_{\|x\|_1=n_k} (E_{n_k}^x)^\mathrm{c}\}} \\
&\stackrel{{\mbox{\footnotesize{\eqref{conditional_F_bd}}}}}{\geq} \bigg(1-32\sum_{k=k_1}^\infty \frac{n_k^{2}\e^{n_k}}{3^{n_k}}\bigg)\prod_{k=k_0}\one_{\{\bigcap_{\|x\|_1=n_k} (E_{n_k}^x)^\mathrm{c}\}} \\
&\stackrel{{\mbox{\footnotesize{\eqref{sum_with_n_k}}}}}{\geq} \frac{7}{8}\one_{G_0}.
}
It now follows from \eqref{G_bd} that
\eeq{ \label{Gprime_bd}
\P(G) \geq \frac{7}{8}\P(G_0) \geq \frac{3}{4}.
}
Having chosen $k_1$, we will assume $n$ satisfies
\eeq{ \label{n_vs_k}
\floor{(\log_2 n)/2} \geq k_1+1.
}

\textbf{Part 2.} For each edge $e$, let $\|e\|$ denote its distance from the origin, i.e.~the graph distance from $0$ to the closest endpoint of $e$.
For each $e$ with $\|e\| \leq n$, set
\eeq{ \label{edge_eps_def}
\eps_e \coloneqq \frac{\alpha}{(\|e\|+1)\sqrt{\log n}},
}
where $\alpha$ is a constant to be chosen below.
For each such $e$, define $\wt X_e$ as in \eqref{new_X} with $\eps = \eps_e$ and
$X_e'$ given in \eqref{prime_definitions}.
Let $\wt T_n = \wt T(0,y_n)$ be the passage time if $X_e$ is replaced by $\wt X_e$ whenever $\|e\|\leq n$.
Because there are at most $C(i+1)$ edges $e$ with $\|e\| = i$, we have
\eq{
\sum_{\|e\| \leq n} \eps_e^2 = \frac{\alpha^2}{\log n}\sum_{i=1}^n\sum_{\|e\|=i} \frac{1}{(i+1)^2}
\leq \frac{\alpha^2}{\log n} \sum_{i=1}^n \frac{C}{i+1} \leq C\alpha^2.
}
Hence, by Lemma \ref{tv_lemma},
\eq{
d_{\tv}(\l_{T_n},\l_{\wt T_n}) \leq C\alpha.
}
Choose $\alpha$ so that
\eeq{
d_{\tv}(\l_{T_n},\l_{\wt T_n}) \leq \frac{1}{4}.\label{FPP_final_1}
}
Now we aim to show that with sufficiently large probability,
$T_n-\wt T_n$ is of order $\sqrt{\log n}$.
Let $e_1,\dots,e_N$ be a geodesic from $0$ to $y_n$, chosen according to same deterministic rule as before. 
Note that necessarily $N\geq \|y\|_1 \geq n$.
We will use the notation $e_i = (x_{i-1},x_i)$ to denote endpoints of $e_i$ in the order traversed by the geodesic.
For each $k = 1,\dots,\floor{(\log_2 n)/2}$, let $i_k$ be the first index such that $\|x_{i_k}\|_1 = n_{2k}$, where the $n_k$'s were chosen in Part 1 and satisfy $2^{k-1}<n_k\leq2^k$.
Observe that
\eeq{ \label{edge_distance_bd}
\|e_i\| \leq n_{2k}-1\leq 4^k-1 \quad \text{for every $i\leq i_{k}$}.
}
Furthermore, $(e_{{i_k}+1},\dots,e_{i_{k+1}})$ is a geodesic from $x_{i_k} \in B_{n_{2k}}(0)$ to $x_{i_{k+1}} \in B_{n_{2k+2}}(0)$, where $n_{2k+2} > 2^{2k+1} \geq 2n_k$.
Therefore, on the event $G$ defined in \eqref{G_def},
\eq{
\sum_{i=i_k+1}^{i_{k+1}} W_{e_i} 
\geq cn_{2k} > c2^{2k-1} \quad \text{for all } k = k_1/2,\dots,\floor{(\log_2 n)/2}-1.
}
which implies
\eeq{ \label{G_consequence}
\sum_{i=1}^N \eps_{e_i}W_{e_i} 
\geq \sum_{i=i_{k_1/2}+1}^N \eps_{e_i}W_{e_i} 
&\stackrel{\phantom{\mbox{\footnotesize\eqref{edge_distance_bd}}}}{\geq} \sum_{k=k_1/2}^{\floor{(\log_2 n)/2}-1} \sum_{i=i_k+1}^{i_{k+1}} \eps_{e_i}W_{e_i} \\
&\stackrel{\mbox{\footnotesize\eqref{edge_distance_bd}}}{\geq} \sum_{k=k_1/2}^{\floor{(\log_2 n)/2}-1} \frac{\alpha}{4^{k+1}\sqrt{\log n}} \sum_{i=i_k+1}^{i_{k+1}} W_{e_i} \\
&\stackrel{\phantom{\mbox{\footnotesize\eqref{edge_distance_bd}}}}{\geq} \frac{\alpha}{\sqrt{\log n}}\sum_{k=k_1/2}^{\floor{(\log_2 n)/2}-1}\frac{c2^{2k-1}}{4^{k+1}} \\
&\stackrel{\phantom{\mbox{\footnotesize\eqref{edge_distance_bd}}}}{=} \frac{\alpha c(\floor{(\log_2 n)/2}-k_1/2)}{8\sqrt{\log n}}\\
&\stackrel{\mbox{\footnotesize\eqref{n_vs_k}}}{\geq} \frac{\alpha c\sqrt{\log n}}{16\log 2} \eqqcolon \theta\sqrt{\log n}. \raisetag{5\baselineskip}
}
Denote by $\sigma(X,X^{(1)},\dots,X^{(m)})$ the $\sigma$-algebra generated by the $X_e$'s and $X_e^{(j)}$'s, $1\leq j\leq m$.
Recall that each $\wt X_{e_i}$ is equal to $\min(X_{e_i},X^{(1)}_{e_i},\dots,X^{(m)}_{e_i}) = X_{e_i}-Z_{e_i}$ independently with probability $\eps_{e_i}$, and equal to $X_{e_i}$ otherwise.
In the former case, the value of $\wt T_n$ is lowered relative to $T_n$ by at least $Z_{e_i}$; in the latter case, no change occurs.
Therefore,
\eq{
T_n - \wt T_n\geq \sum_{i=1}^N \one_{\{Y_{e_i}=1\}}Z_{e_i} \eqqcolon D,
}
where the $Y_{e_i}$'s are Bernoulli($\eps_{e_i}$) random variables independent of each other and independent of $\sigma(X,X^{(1)},\dots,X^{(m)})$.
It follows that for any $t\geq0$,
\eq{
\P\givenp[\big]{D\leq t}{\sigma(X,X^{(1)},\dots,X^{(m)})}
&\leq \e^{t}\E\givenp[\big]{\e^{-D}}{\sigma(X,X^{(1)},\dots,X^{(m)})} \\
&= \e^{t}\prod_{i=1}^N (1-\eps_{e_i} + \eps_{e_i}\e^{-Z_{e_i}}) \\
&\leq \e^{t}\prod_{i=1}^N \exp\{-\eps_{e_i}(1-\e^{-Z_{e_i}})\} \\
&= \e^{t}\exp\bigg\{-\sum_{i=1}^N \eps_{e_i}W_{e_i}\bigg\}.
}
Therefore, on the event $G$, \eqref{G_consequence} shows
\eq{
\P\givenp[\Big]{D\leq\frac{\theta}{2}\sqrt{\log n}}{\sigma(X,X^{(1)},\dots,X^{(m)})}\one_G \leq \e^{-\frac{\theta}{2}\sqrt{\log n}}.
}
Assuming $n$ is large enough that
\eeq{ \label{n_easy}
\e^{-\frac{\theta}{2}\sqrt{\log n}} \leq \frac{1}{2},
}
we have
\eeq{ \label{single_path}
\P\givenp[\Big]{D>\frac{\theta}{2}\sqrt{\log n}}{\sigma(X,X^{(1)},\dots,X^{(m)})} \geq \frac{1}{2}\one_G,
}
and thus
\eeq{ \label{FPP_final_2}
\P\Big(T_n - \wt T_n > \frac{\theta}{2}\sqrt{\log n}\Big) \geq \P\Big(D>\frac{\theta}{2}\sqrt{\log n}\Big)\geq \frac{1}{2}\P(G)\stackrel{\mbox{\footnotesize{\eqref{Gprime_bd}}}}{\geq}\frac{3}{8}.
}
Using \eqref{FPP_final_1} and \eqref{FPP_final_2} in Lemma \ref{fluctuation_lemma}, we see that $T_n$ has fluctuations of order at least $\sqrt{\log n}$.
\end{proof}

\subsection{Proof of Theorem \ref{fpp_thm_2}}
Recall Definition \ref{curvature_def} for a direction of curvature, as well as the exponent $\chi'$ from \eqref{chi_prime_def}.

\textbf{Part 1.}
Fix any unit vector $x$ that is a direction of curvature for $\b$, and fix any $\delta > 0$.
We will write $T_n = T(0,[nx])$, where $[y]$ denotes the unique element of $\Z^2$ such that $y\in[y]+[0,1)^d$.
Let $L$ be the line passing through $0$ and $x$, and let $\Lambda_n$ be the cylinder of width $n^{3/4+\delta}$ centered about~$L$:
\eq{
\Lambda_n \coloneqq \{z \in \Z^2: d(z,L) \leq n^{3/4+\delta}\},
}
where $d(z,L) = \inf\{\|z-y\|_2 : y\in L\}$.
Under the given assumptions, \cite[Theorem 1.2]{damron-kubota16} guarantees $\chi' \leq 1/2$.
It then follows from \cite[Theorem 6 and (2.21)]{newman-piza95} that there exists $q_0 \in (0,1]$ such that with probability at least $q_0$, the following event, which we call $G_1$, is true:
For all large $n$, all geodesics from the origin to $[nx]$ lie entirely inside~$\Lambda_n$.

We would like to replace $\Lambda_n$ with a finite set.
To do so, we let $L_n$ be the line segment connecting $0$ and $nx$, and then introduce
\eq{
V_n \coloneqq \{z \in \Z^2 : d(z,L_n) \leq n^{3/4+2\delta}\}.
}
Suppose toward a contradiction that $G_1$ occurs but there exists a geodesic from $0$ to $[nx]$ that remains inside $\Lambda_n$ but not $V_n$.
Observe that from any $z \in \Lambda_n \setminus V_n$, the closest point on $L_n$ is either $0$ or $[nx]$.
Consequently, it follows from our supposition that from one of the endpoints of $L_n$ (say $0$, for concreteness), there are points $z_1$ within distance $n^{3/4+\delta}$ and $z_2$ at distance at least $n^{3/4+2\delta}$, such that $T(0,z_1) \geq T(0,z_2)$; see Figure \ref{fpp_fig}.
By the shape theorem, this inequality can only happen for finitely many $n$.
From this argument we conclude that with probability at least $q_0$, the following event, which we call $G_2$, is true:
For all large $n$, all geodesics from the origin to $[nx]$ lie entirely inside $V_n$.

\begin{figure}[h!]
\centering
\includegraphics[width=0.8\textwidth]{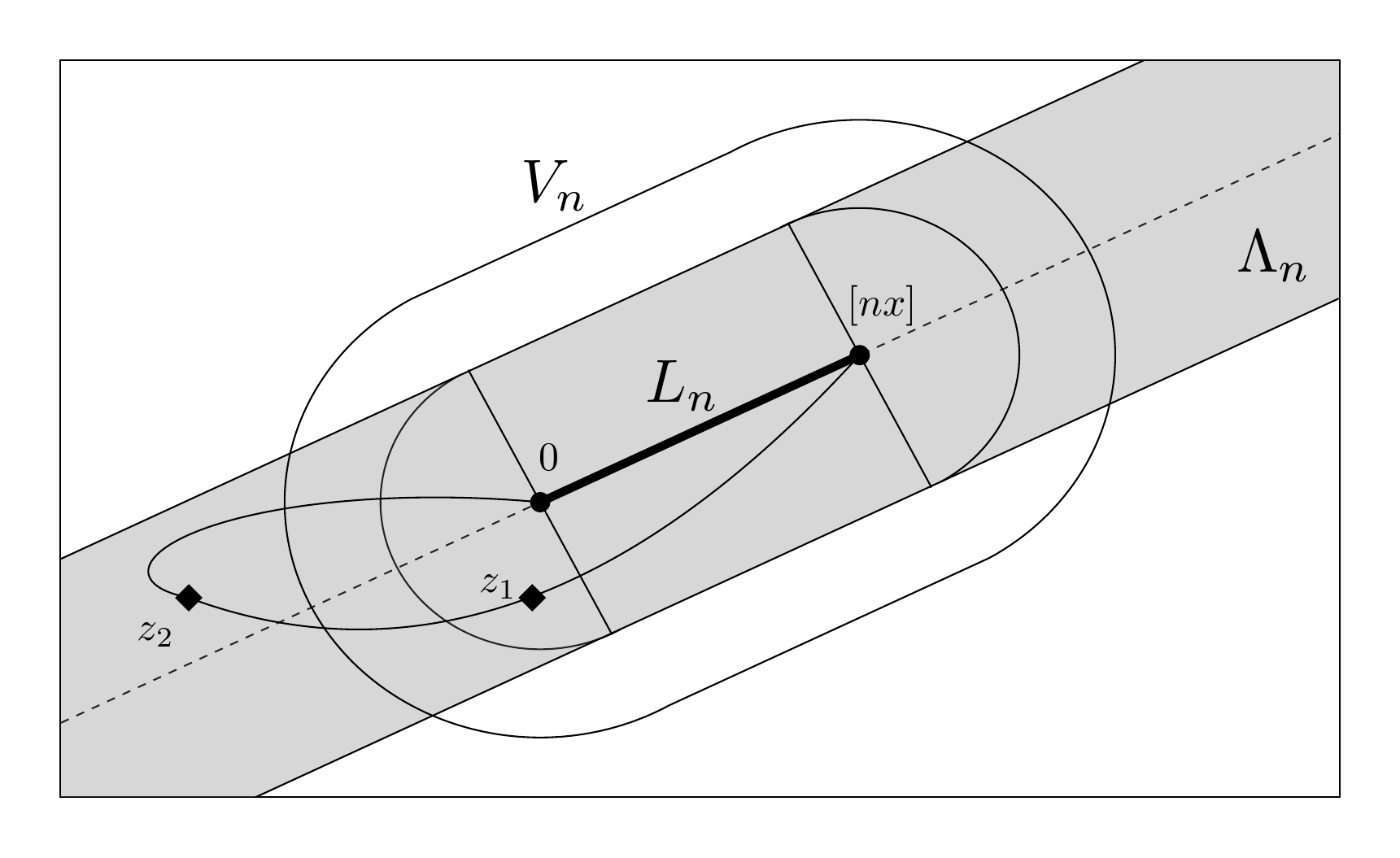}
\caption{The geodesic connecting $0$ and $[nx]$ remains inside $\Lambda_n$ but exits and re-enters $V_n$.
The point $z_2$ is outside $V_n$ but has a shorter passage time to $0$ than does $z_1$, which is within distance $n^{3/4+\delta}$ of $0$.} \label{fpp_fig}
\end{figure}

Note that \eqref{fpp_assumption_2b} is implied by $\E(X_e^{1/2}) < \infty$ and thus also by $\E(X_e^\lambda)<\infty$ for $\lambda>3/2$.
From Lemma \ref{fpp_lemma} we can find $\delta > 0$ and $\rho \in (0,1)$ such that \eqref{sum_n_bd_1} holds.
As in the proof of Theorem \ref{fpp_thm}, for each edge $e \in  E(\Z^2)$ we define 
\eq{
X_e' \coloneqq \min(X_e,X_e^{(1)},\dots,X_e^{(m)}), \qquad Z_e \coloneqq X_e-X_e', \qquad W_e \coloneqq 1-\e^{-Z_e}.
}
When considering geodesics between $0$ and $[nx]$, we always choose a distinguished geodesic \linebreak $(e_1,\dots,e_N)$ according some deterministic rule.
As in the proof of Theorem~\ref{fpp_thm}, we take $m$ large enough and $c>0$ small enough that 
\eq{
\P\givenp[\bigg]{\sum_{i=1}^N W_{e_i} \leq cn}{\sigma(X)}\one_{(E_n^0)^\mathrm{c}} \leq \frac{\e^n}{3^n} \quad \text{for all $n\geq1$.}
}
Let $F_n$ be the event that $\sum_{i=1}^N W_{e_i} \leq cn$ (here we have fixed the endpoints, and so this event is different from the $F_n$ considered in the proof of Theorem \ref{fpp_thm}). 
By the above display and \eqref{sum_n_bd_1}, there is $n_0$ such that
\eeq{ \label{n_assumption_1}
\P(F_n) \leq \frac{q_0}{2} \quad \text{for all $n\geq n_0$}.
}

\textbf{Part 2.}
Now we set
\eq{
\eps \coloneqq \alpha n^{-7/8-\delta},
}
where $\alpha$ will be chosen below, and define the perturbed edge weights as in \eqref{new_X}:
For each edge $e$ with both endpoints in $V_n$, we let
\eq{
\wt X_e = \begin{cases}
X_e' &\text{if } Y_e = 1, \\
X_e &\text{if } Y_e = 0,
\end{cases}
\quad \text{where} \quad
Y_e \stackrel{\text{i.i.d.}}{\sim} \mathrm{Bernoulli}(\eps).
}
Denote by $\wt T_n$ be the passage time from $0$ to $[nx]$ if $X_e$ is replaced by $\wt X_e$ whenever $e$ has both endpoints in $V_n$.
Before proceeding, let us note that by Lemma \ref{tv_lemma},
\eq{
d_{\tv}(\l_T,\l_{\wt T}) \leq C\alpha n^{-7/8-\delta}\sqrt{\#(\text{edges in $V_n$})} \leq C\alpha n^{-7/8-\delta}\sqrt{Cn^{7/4+2\delta}} = C\alpha,
}
where $C$ depends only on $\l_X$ and $m$.
We can then take $\alpha$ sufficiently small that
\eeq{ \label{alpha_choice}
d_{\tv}(\l_T,\l_{T'}) \leq \frac{q_0}{8}.
}
We will also assume
\eeq{ \label{n_assumption_2}
\frac{\alpha c}{2}n^{1/8-\delta} \geq -\log\Big(\frac{q_0}{4}\Big).
}
Let $(e_1,\dots,e_N)$ be the distinguished geodesic from $0$ to $[nx]$, which lies entirely inside $V_n$ for all large $n$ provided $G_2$ occurs.
In this case, as in the proof of Theorem~\ref{fpp_thm},
\eq{
T_n - \wt T_n\geq \sum_{i=1}^N \one_{\{Y_{e_i}=1\}}Z_{e_i} \eqqcolon D,
}
where the $Y_{e_i}$'s are i.i.d.~Bernoulli($\eps$) random variables that are independent of $\sigma(X,X^{(1)},\dots,X^{(m)})$.
So on the event $F_n^\mathrm{c} \cap G_2$, for any $t>0$,
\eq{
\P\givenp{D\leq tn^{1/8-\delta}}{\sigma(X,X^{(1)},\dots,X^{(m)})}\one_{F_n^\mathrm{c} \cap G_2}
&\leq \e^{tn^{1/8-\delta}}\E\givenp{\e^{-D}}{\sigma(X,X^{(1)},\dots,X^{(m)})}\one_{F_n^\mathrm{c} \cap G_2} \\
&= \one_{F_n^\mathrm{c} \cap G_2}\e^{tn^{1/8-\delta}}\prod_{i=1}^N (1-\eps + \eps \e^{-Z_{e_i}}) \\
&\leq \one_{F_n^\mathrm{c} \cap G_2}\e^{tn^{1/8-\delta}}\prod_{i=1}^N \exp\{-\eps(1-\e^{-Z_{e_i}})\} \\
&=\one_{F_n^\mathrm{c} \cap G_2}\e^{tn^{1/8-\delta}}\exp\bigg\{-\eps\sum_{i=1}^N W_{e_i}\bigg\} \\
&\leq \e^{tn^{1/8-\delta}-\alpha c n^{1/8-\delta}}.
}
Choosing $t = \alpha c/2$, we find that
\eq{
\P\Big(T_n - \wt T_n \leq \frac{\alpha c}{2} n^{1/8-\delta}\Big) &\stackrel{\phantom{\mbox{\footnotesize{\eqref{n_assumption_1},\eqref{n_assumption_2}}}}}{\leq} \P(F_n \cup G_2^\mathrm{c})+\e^{-\frac{\alpha c}{2} n^{1/8-\delta}} \\
&\stackrel{\mbox{\footnotesize{\eqref{n_assumption_1},\eqref{n_assumption_2}}}}{\leq} \frac{q_0}{2} + 1 - q_0 + \frac{q_0}{4}
= 1 - \frac{q_0}{4}.
}
Together with \eqref{alpha_choice} and Lemma \ref{fluctuation_lemma}, this completes the proof.

\subsection{Proof of Theorem \ref{lpp_thm}}  We begin with a lemma that will serve a similar purpose as Lemma \ref{fpp_lemma} did in the proof of Theorem \ref{fpp_thm}.

\begin{lemma} \label{lpp_lemma}
Consider directed site percolation on $\Z^2_+$ in which each site is open independently with probability $p < \vec p_{c,\,\mathrm{site}}(\Z^2)$.
Given $\rho > 0$, let $E_n$ be the event that exists a directed path $(v_0,v_1,\dots,v_n)$ with $\|v_0\|_1 \leq n$, such that
\eq{
\#\{1 \leq i \leq n : v_i \mathrm{\ closed}\} < \rho n.
}
Then there is $\rho$ sufficiently small that for some $a,b>0$,
\eq{
\P(E_n) \leq a\e^{-bn} \quad \text{for all $n\geq1$.}
}
\end{lemma}

\begin{proof}
First observe that by a union bound,
\eq{
\P(E_n) \leq \frac{(n+1)(n+2)}{2}\P(E_n^0),
}
where $E_n^0$ is the event that there exists a directed path of length $n$ starting at the origin and passing through fewer than $\rho n$ closed sites.
If we can prove $\P(E_n^0) \leq a\e^{-bn}$ for some $a,b>0$, then it will follow that $\P(E_n) \leq a'\e^{-b'n}$ for some $a',b'>0$.
Therefore, we henceforth concern ourselves only with the event $E_n^0$.

For a directed path $\vec\gamma = (\gamma(0),\gamma(1),\dots,\gamma(\ell))$, let $|\vec\gamma| = \ell$ denotes its length.
Let $A_k$ be the event that there exists an open directed path of length $k$ starting at the origin.
Since $p < \vec p_{c,\,\mathrm{site}}(\Z^2)$, \cite[Theorem 7]{griffeath80} (see also \cite[Theorem 14]{durrett-liggett81}) guarantees the existence of $c_1,c_2>0$ such that
\eq{
\P(A_k) \leq c_1\e^{-c_2k} \quad \text{for all $k\geq1$.}
}
Choose $k$ large enough that
\eeq{ \label{choosing_k}
\P(A_k) \leq \frac{1}{36(k+1)^2},
}
and then set $\rho \coloneqq 1/(4k)$. 
Let $F_n$ be the event that some directed path of length $nk$ starting at the origin passes through fewer than $n/2$ closed sites.
Since $\rho(n+1)k = (n+1)/4 \leq n/2$ for any $n\geq 1$, we have the following containments for $n\geq1$ and $0\leq j<k$:
\eq{
E_{nk+j}^0 &= \{\exists\ \vec\gamma,\ \vec\gamma(0)=0,\ |\vec\gamma| = nk+j, \text{ with fewer than $\rho(nk+j)$ closed sites}\} \\
&\subset \{\exists\ \vec\gamma,\ \vec\gamma(0)=0,\ |\vec\gamma| = nk, \text{ with fewer than $\rho(n+1)k$ closed sites}\} \\
&\subset \{\exists\ \vec\gamma,\ \vec\gamma(0)=0,\ |\vec\gamma| = nk, \text{ with fewer than $n/2$ closed sites}\} = F_n.
}
It suffices, then, to obtain a bound of the form $\P(F_n) \leq a\e^{-bn}$.
The remainder of the proof is to achieve such an estimate.

Consider the set
\eq{
\Lambda_n \coloneqq \{\vc w = (w_0=0,w_1,\cdots,w_n)\ |\ \forall\ i=1,\dots,n,\ \exists\ \vec \gamma : w_{i-1}\to w_i \text{ with } |\vec\gamma| = k\}.
}
In words, $\Lambda_n$ is the set of all $(n+1)$-tuples whose $i^\text{th}$ coordinate is $ik$ steps from the origin, and for which there exists a directed path passing through all its coordinates.
Since a directed path of length $\ell$ starting at a fixed position must terminate at one of exactly $\ell+1$ vertices, the cardinality of $\Lambda_n$ is
\eeq{ \label{counting_Lambda}
|\Lambda_n| = (k+1)^n.
}
Recall that $\vec\Gamma_{nk}$ denotes the set of directed paths of length $nk$ starting at the origin.
For each $\vc w \in \Lambda_n$, let $\vec\Gamma_{\vc w}$ denote the subset of those paths traversing the coordinates of $\vc w$:
\eq{
\vec\Gamma_{\vc w} \coloneqq \{\vec\gamma \in \vec\Gamma_{nk} : \vec\gamma(ik) = w_i, 1\leq i\leq n\}.
}
From the definitions, we have $\vec\Gamma_{nk} = \bigcup_{\vc w \in \Lambda_n} \vec\Gamma_{\vc w}$.
Moreover, if we define $F_{\vc w}$ to be the event that some $\vec\gamma\in\vec\Gamma_{\vc w}$ has fewer than $n/2$ closed sites, then
\eeq{ \label{event_partition}
F_n = \bigcup_{\vc w \in\Lambda_n} F_{\vc w}.
}
Fix any $\vc w \in \Gamma_n$.
For $1\leq i\leq n$, let $X_i$ denote the minimum number of closed sites in a directed path of length $k$ starting at $w_{i-1}$.
It is immediate from translation invariance that $\P(X_i \geq 1) = 1-\P(A_k)$.
We thus have the estimate
\eq{
\P(F_{\vc w}) &\leq \P(X_1 + \cdots + X_n \leq n/2) \\
&\leq \P(\one_{\{X_1 \geq 1\}} + \cdots + \one_{\{X_n \geq 1\}} \leq n/2) \\
&= \sum_{i = 0}^{\floor{n/2}} {n \choose i} \big(1-\P(A_k)\big)^{i}\P(A_k)^{n-i} \\
&\leq \frac{n}{2} {n \choose \floor{n/2}} \P(A_k)^{n/2} \leq C\sqrt{\frac{n}{2\pi}}\Big(2\sqrt{\P(A_k)}\Big)^n,
}
where the final inequality holds for some $C>0$ by Stirling's approximation.
It now follows from \eqref{counting_Lambda}, \eqref{event_partition}, and \eqref{choosing_k} that
\eq{
\P(F_n) \leq C\sqrt{\frac{n}{2\pi}}\Big(2(k+1)\sqrt{\P(A_k)}\Big)^n \leq C\sqrt{\frac{n}{2\pi}}3^{-n} \leq a2^{-n}
}
for some $a>0$. 
\end{proof}

\begin{proof}[Proof of Theorem \ref{lpp_thm}] \hspace{1in}\\[-0.5\baselineskip]

\textbf{Part 1.}
For each $v \in \Z^2_+\setminus\{0\}$, define
\eq{
X_v' \coloneqq \max(X_v,X_v^{(1)},\dots,X_v^{(m)}), \qquad Z_v \coloneqq X_v' - X_v, \qquad W_v \coloneqq 1 - \e^{-Z_v},
}
where $m$ is chosen below, and $(X_v^{(j)})_{v\in\Z^2_+}$, $1\leq j\leq m$ are independent copies of the i.i.d.~vertex weights.
Recall that $S = \esssup X_v$.
If $S = \infty$, take $\delta = 1$ and choose $S'$ sufficiently large that
\eq{ 
\P(X_v \geq S' - 2\delta) < \vec p_{c,\,\mathrm{site}}(\Z^2).
}
If $S < \infty$, set $S' = S$ and choose $\delta > 0$ sufficiently small that the above display holds.
In either case, we can find $m$ sufficiently large that
\eq{
\P(\max(X_v^{(1)},\dots,X_v^{(m)}) < S'-\delta) < \vec p_{c,\,\mathrm{site}}(\Z^2) - \P(X_v \geq S' - 2\delta),
}
so that
\eq{
\P(Z_v < \delta)
&\leq \P(\max(X_v^{(1)},\dots,X_v^{(m)})< S'-\delta) + \P(X_v \geq S-2\delta) < \vec p_{c,\,\mathrm{site}}(\Z^2).
}
By Lemma \ref{lpp_lemma}, there is $\rho \in (0,1)$ and $a,b>0$ so that with probability at least $1 - a\e^{-b2^k}$, every directed path $(v_0,v_1,\dots,v_{2^k})$ of length $2^k$ with $\|v_0\|_1 = 2^k$ satisfies
\eq{
\sum_{i=1}^{2^k} W_{v_i} \geq \rho (1-\e^{-\delta})2^k.
}
Let $G$ be the event that this is the case for every $k \geq k_1$, where $k_1$ is chosen large enough that
\eeq{ \label{Gprime_bd_lpp}
\P(G) \geq 3/4.
}
We will assume $n$ is large enough to satisfy \eqref{n_vs_k}.

\textbf{Part 2.}
Similarly to \eqref{edge_eps_def}, we will take
\eq{
\eps_v \coloneqq \frac{\alpha}{\|v\|_1\sqrt{\log n}}, \quad v \in \Z^2_+ \setminus \{0\},
}
and define $\wt X_v$ as in \eqref{new_X} with $\eps = \eps_v$.
Let $T_n = T(0,y_n)$ be the passage time with the $X_v$'s as the vertex weights, and let $\wt T_n = \wt T(0,y_n)$ be the passage time with the $\wt X_v$'s.
The constant $\alpha > 0$ is taken small enough that \eqref{FPP_final_1} holds.

On the event $G$, every directed path $(0=v_0,v_1,\dots,v_n)$ of length $n$ satisfies
\eeq{ \label{Gprime_consequence}
\sum_{i=1}^n \eps_{v_i}W_{v_i} 
\geq \sum_{i=2^{k_1}}^n \eps_{v_i}W_{v_i} 
&\stackrel{\phantom{\mbox{\footnotesize\eqref{n_vs_k}}}}{\geq} \sum_{k=k_1}^{\floor{\log_2 n}-1} \sum_{i=2^{k}+1}^{2^{k+1}} \eps_{v_i}W_{v_i} \\
&\stackrel{\phantom{\mbox{\footnotesize\eqref{n_vs_k}}}}{\geq} \sum_{k=k_1}^{\floor{\log_2 n}-1} \frac{\alpha}{2^{k+1}\sqrt{\log n}} \sum_{i=2^{k}+1}^{2^{k+1}} W_{v_i} \\
&\stackrel{\phantom{\mbox{\footnotesize\eqref{n_vs_k}}}}{\geq} \frac{\alpha}{\sqrt{\log n}}\sum_{k=k_1}^{\floor{\log_2 n}-1}\frac{\rho(1-\e^{-\delta})2^k}{2^{k+1}} \\
&\stackrel{\phantom{\mbox{\footnotesize\eqref{n_vs_k}}}}{=} \frac{\alpha \rho(1-\e^{-\delta})(\floor{\log_2 n}-k_1)}{2\sqrt{\log n}} \\
&\stackrel{\mbox{\footnotesize\eqref{n_vs_k}}}{\geq} \frac{\alpha \rho(1-\e^{-\delta})\sqrt{\log n}}{4\log 2} \eqqcolon \theta\sqrt{\log n}. \raisetag{5\baselineskip}
}
The argument is now completed by proceeding exactly as in the proof of Theorem \ref{fpp_thm} following \eqref{G_consequence}, where \eqref{Gprime_bd} and \eqref{G_consequence} are replaced by \eqref{Gprime_bd_lpp} and \eqref{Gprime_consequence}, respectively.
\end{proof}

\subsection{Proof of Theorem \ref{dp_thm}}
We will absorb the inverse temperature $\beta$ into the $X_v$'s and then work in the case $\beta = 1$.
Let the notation be as in the proof of Theorem \ref{lpp_thm}.
In addition, let $\wt H_n$ and $\wt Z_n$ be the Hamiltonian and partition function, respectively, in the environment formed by the $\wt X_v$'s.
Now \eqref{FPP_final_1} reads as
\eeq{
d_{\tv}(\l_{\log Z_n},\l_{\log \wt Z_n}) \leq \frac{1}{4}.\label{dp_final_1}
}
We repeat all steps of the proof of Theorem \ref{lpp_thm} and take $n$ sufficiently large that on the event $G$ defined therein, 
\eeq{ \label{good_pathwise}
\P\givenp[\Big]{\wt H_n(\vec\gamma)-H_n(\vec\gamma) \geq \frac{\theta}{2}\sqrt{\log n}}{\sigma(X,X^{(1)},\dots,X^{(m)})} \geq \frac{3}{4}\one_G
}
for \textit{every} $\vec\gamma\in\vec\Gamma_n$.
(This is in analogy with \eqref{single_path}, but for $n$ satisfying a more restrictive lower bound than \eqref{n_easy}.)
The remainder of the argument must be slightly modified to account for the fact that all paths contribute to the free energy, not just those with maximum weight.

For each $\vec\gamma\in\Gamma_n$, define
\eq{
D_{\vec\gamma} \coloneqq \begin{cases} 
\frac{\theta}{2}\sqrt{\log n} &\text{if } \wt H_n(\vec\gamma)-H_n(\vec\gamma) \geq \frac{\theta}{2}\sqrt{\log n}, \\
0 &\text{otherwise}.
\end{cases}
}
From Jensen's inequality, It is immediate that
\eq{
\log \wt Z_n - \log Z_n = \log \sum_{\vec\gamma\in\vec\Gamma_n} \frac{\e^{H_n(\vec\gamma)}}{Z_n} \e^{\wt H_n(\vec\gamma)-H_n(\vec\gamma)}
&\geq \log \sum_{\vec\gamma\in\vec\Gamma_n} \frac{\e^{H_n(\vec\gamma)}}{Z_n} \e^{D_{\vec\gamma}}
\geq \sum_{\vec\gamma\in\vec\Gamma_n} \frac{\e^{H_n(\vec\gamma)}}{Z_n} D_{\vec\gamma}.
}
On one hand,
\eeq{ \label{lower_bound_expectation}
\E\bigg[\sum_{\vec\gamma\in\vec\Gamma_n} \frac{\e^{H_n(\vec\gamma)}}{Z_n} D_{\vec\gamma}\bigg] 
&\stackrel{\phantom{\mbox{\footnotesize\eqref{good_pathwise}}}}{=} \E\bigg[\sum_{\vec\gamma\in\vec\Gamma_n} \frac{\e^{H_n(\vec\gamma)}}{Z_n} \E\givenp[\big]{D_{\vec\gamma}}{\sigma(X,X^{(1)},\dots,X^{(m)})}\bigg] \\
&\stackrel{\mbox{\footnotesize\eqref{good_pathwise}}}{\geq} \E\bigg[\sum_{\vec\gamma\in\vec\Gamma_n} \frac{\e^{H_n(\vec\gamma)}}{Z_n} \one_G\frac{3\theta}{8}\sqrt{\log n}\bigg] \\
&\stackrel{\phantom{\mbox{\footnotesize\eqref{good_pathwise}}}}{=} \frac{3\theta}{8}\sqrt{\log n}\, \P(G) \stackrel{\mbox{\footnotesize{\eqref{Gprime_bd_lpp}}}}{\geq} \frac{9\theta}{32}\sqrt{\log n}.
}
On the other hand, we have the deterministic upper bound
\eq{
\sum_{\vec\gamma\in\vec\Gamma_n} \frac{\e^{H_n(\vec\gamma)}}{Z_n} D_{\vec\gamma} \leq \frac{\theta}{2}\sqrt{\log n}.
}
Therefore, the lower bound \eqref{lower_bound_expectation} can only hold if
\eq{
\P\Big(\log \wt Z_n - \log Z_n \geq \frac{\theta}{16}\sqrt{\log n}\Big) 
&\geq \P\bigg(\sum_{\vec\gamma\in\vec\Gamma_n} \frac{\e^{H_n(\vec\gamma)}}{Z_n} D_{\vec\gamma} \geq \frac{\theta}{16}\sqrt{\log n}\bigg) 
\geq \frac{1}{2}.
}
Together with \eqref{dp_final_1} and Lemma \ref{fluctuation_lemma}, this completes the proof.

\section{Acknowledgments}
We thank Antonio Auffinger, Francisco Arana Herrera, Si Tang, and Nhi Truong for useful conversations, and an anonymous referee for several corrections and valuable suggestions.


\bibliography{LB_fluctuations}

\end{document}